\documentclass[pamm,a4paper,fleqn]{w-art}
\usepackage{times,cite,w-thm}
\usepackage[T1]{fontenc}
\usepackage[utf8]{inputenc}

\usepackage{amsfonts}
\usepackage{mathtools}

\usepackage{verbatim}

\theoremstyle{plain}

\theoremstyle{definition}

\usepackage{graphicx}

\newcommand{\C}{\mathbb{C}}

\newcommand{\N}{\mathbb{N}}
\newcommand{\HH}{\mathcal{H}}
\newcommand{\D}{\mathbb{D}}
\newcommand{\HHH}{\mathbb{H}}
\newcommand{\SSS}{\mathbb{S}}
\newcommand{\Id}{I}
\DeclareMathOperator{\Span}{span}

\definecolor{myBlue}{RGB}{30,144,255} 
\definecolor{myGreen}{RGB}{69,169,0} 
\definecolor{myRed}{RGB}{165,12,42}
\definecolor{myOrange}{RGB}{225,92,22}
\definecolor{dgreen}{rgb}{0,.8,0.2}
\newcommand{\aaa}[1]{{\color{black} #1}}

\newcommand{\se}[1]{{\color{black} #1}}

\begin{document}

\TitleLanguage[EN]
\title{Connection of hypocoercivity and hypocontractivity via the $\theta$--methods}


\author{\firstname{Anton}  \lastname{Arnold}\inst{1,}%
  \footnote{e-mail \ElectronicMail{anton.arnold@tuwien.ac.at}}}

\address[\inst{1}]{\CountryCode[AT]TU Wien, Institute of Analysis and Scientific Computing, Wiedner Hauptstr. 8-10, A-1040 Wien, Austria}
\author{\firstname{Stefan} \lastname{Egger}\inst{1,}%
  \footnote{Corresponding author: e-mail \ElectronicMail{stefan.egger@tuwien.ac.at}}}

\AbstractLanguage[EN]
\begin{abstract}
Recent literature shows that hypocoercivity properties of linear evolution equations (in particular their exponential decay and the sharp short time decay of their propagator norm) carry over to their discretization via the midpoint rule. This note discusses this connection for the (other) $\theta$--methods, i.e.\ for $\theta\ne\frac12$. 

It is shown that any implicit discretization with $\theta\in (\frac12,1]$ (pertaining to a hypocoercive continuous-time evolution equation) is contractive, and not only hypocontractive -- in contrast to the midpoint rule. For a coercive continuous-time evolution equation, a discretization with $\theta\in [0,\frac12)$ is contractive for time steps small enough.
\end{abstract}
\maketitle                   

\section{Introduction}

This note is concerned with the comparison of the short- and large-time behavior of linear evolution equations to the related behavior of certain time discretizations. In particular we are concerned with \emph{semi-dissipative} systems in a separable Hilbert space $\HH$:
\begin{equation}\label{ODE}
    \dot x=-Bx,\quad x(0)=x_0,\quad t>0,
\end{equation}
with some bounded linear operator $B\in \mathcal{B}(\HH)$, where its Hermitian part satisfies $B_H:=\frac12(B+B^*)\ge0$. Moreover, we shall assume that $B$ has a trivial kernel \se{in order to simplify the definition of hypocoercivity. Note that this can always be achieved by restricting the space $\HH$ to $\ker(B)^\perp$ (see also §I of \cite{Vil09})}. 
As a discrete counterpart we shall consider here the $\theta$-methods with $\theta\in[0,1]$ and step size $\tau>0$, i.e.:
\begin{equation}\label{theta1}
    \frac{x_{k+1}-x_k}{\tau}=-B\big(\theta x_{k+1}+(1-\theta)x_k\big), \quad k\in\N_0,
\end{equation}
or after rearrangement
\begin{equation}\label{theta2}
    x_{k+1}=Dx_k, \quad 
    \mbox{with the operator } D:=(\Id +\tau\theta B)^{-1} (\Id-\tau(1-\theta)B),
    \quad k\in\N_0.
\end{equation}
Our goal is to connect hypocoercivity properties of $B$ in \eqref{ODE} with hypocontractivity properties of $D$ in \eqref{theta2}. In \S4 of \cite{arnold2026connectionhypocoercivityhypocontractivitycayley} this connection was analyzed in detail for the midpoint rule (or Crank-Nicolson scheme), i.e.\ for $\theta=\frac12$. Here we shall extend that discussion to the other $\theta$-methods.

For the continuous-time dynamical system \eqref{ODE} we recall the following definitions from \cite{Vil09, AAC}: 
\begin{definition}
Let $B\in\mathcal{B}(\HH)$.
\begin{enumerate}
    \item[(i)] $B$ is called \emph{hypocoercive} if there exist constants $C\ge1$ and $\lambda>0$ such that the solutions to \eqref{ODE} satisfy
    \begin{equation}
        \forall x_0\in\HH,\quad \forall t\ge0:\quad 
        \|x(t)\| \le C e^{-\lambda t} \|x_0\|,
    \end{equation}
    where $\|\cdot\|$ denotes the norm in $\HH$.
    \footnote{\aaa{If the kernel of $B$ would not be trivial but rather $\Span\{x_\infty\}$, then the required decay estimate would rather read $\|x(t)-x_\infty\| \le C e^{-\lambda t} \|x_0-x_\infty\|$, where the equilibrium $x_\infty$ would be normalized as $x_\infty =P_{\ker(B)}x_0$.}}
    \item[(ii)] Let $-B$ be semi-dissipative, i.e.\ $B_H\ge0$. The \emph{hypocoercivity index (HC-index)} $m_{HC}$ of $B$ is defined as the smallest integer $m\in\N_0$ (if it exists) such that 
    \begin{equation}
        \sum_{j=0}^m (B^*)^j B_H B^j \ge \kappa \Id
    \end{equation}
    for some $\kappa>0$.
\end{enumerate}
\end{definition}

As a key result, this hypocoercivity index characterizes the short-time behavior of the propagator norm for \eqref{ODE} (see \cite[Th. 2.7]{AAC2} for the finite dimensional case and \cite[Th. 4.1]{AchAMN25} for the infinite dimensional case):
\begin{theorem}
Let $-B\in\mathcal{B}(\HH)$ be semi-dissipative. Then, $B$ is hypocoercive (with HC-index $m_{HC}$) if and only if
$$
  \|e^{-Bt}\| = 1-ct^a+ o(t^{a})\quad \mbox{for } t\to0,
$$
with some $a,\,c>0$. In this case, necessarily $a=2m_{HC}+1$.
\end{theorem}

Next we recall for the discrete-time dynamical system \eqref{theta2} the analogous definitions from \cite{AchAM23ELA, arnold2026connectionhypocoercivityhypocontractivitycayley}:
\begin{definition}
For $B\in\mathcal{B}(\HH)$ let $\se{-}\frac{1}{\theta\tau}\not\in \sigma(B)$ \se{if $\theta > 0$}, and hence $D\in\mathcal{B}(\HH)$.
\begin{enumerate}
    \item[(i)] $D$ is called \emph{contractive} if $\|D\|<1$. 
    \item[(ii)] $D$ is called \emph{semi-contractive} if $\|D\|\le1$. 
    \item[(iii)] $D$ is called \emph{hypocontractive} (or \emph{\se{exponentially} stable}) if there exist constants $C\ge1$ and $0<\lambda<1$ such that the solutions to the iteration \eqref{theta2} satisfy
    \begin{equation}
        \forall x_0\in\HH,\quad \forall k\in\N_0:\quad 
        \|x_k\| \le C \lambda^k \|x_0\|.
    \end{equation}
    \item[(iv)] Let $D$ be semi-contractive. Its \emph{hypocontractivity index (dHC-index)} $m_{dHC}$ is defined as the smallest integer $m\in\N_0$ (if it exists) such that 
    \begin{equation}
        \sum_{j=0}^m (D^*)^j (\Id-D^*D) D^j \ge \kappa \Id
    \end{equation}
    for some $\kappa>0$.
\end{enumerate}
\end{definition}

As a key result, this hypocontractivity index characterizes the short-time behavior of the iteration \eqref{theta2} (see \cite[Th. 40]{AchAM23ELA} for the finite dimensional case and \cite[Th. 4.5]{arnold2026connectionhypocoercivityhypocontractivitycayley} for the infinite dimensional case):
\begin{theorem}
Let $D\in\mathcal{B}(\HH)$ be semi-contractive and hypocontractive. Its (finite)  hypocontractivity index is $m_{dHC}\in\N_0$ if and only if
$$
  \|D^j\| = 1\: \mbox{ for all } j=1,...,m_{dHC},\quad \mbox{and } \|D^{m_{dHC}+1}\| < 1.
$$
\end{theorem}

For the midpoint rule, i.e.\ \eqref{theta2} with $\theta=\frac12$, the hypocoercivity properties of $B$ are closely related to the hypocontractivity properties of $D$ (see \cite[\S4]{AchAM23ELA} for the finite dimensional case and \cite[\S4]{arnold2026connectionhypocoercivityhypocontractivitycayley} for the infinite dimensional case):
\begin{theorem}\label{th1.5}
Let $B,\,D\in\mathcal{B}(\HH)$ be well-defined and related via \eqref{theta2}. Then:
\begin{itemize}
    \item[(i)] $B$ is hypocoercive if and only if $D$ is hypocontractive.
    \item[(ii)] $-B$ is semi-dissipative if and only if $D$ is semi-contractive.
    \item[(iii)] Let $B$ be hypocoercive and $-B$ be semi-dissipative. Then, $m_{HC}(B)=m_{dHC}(D)$.
\end{itemize}
\end{theorem}

Here we shall discuss the analogous relation of $B$ and $D$ for $\theta\ne\frac12$. As we shall see in \S\ref{sec:theta}, the clear-cut relation from $\theta=\frac12$ does not extend to the cases $\theta\ne\frac12$. 
In \S\ref{sec:ex} we shall illustrate our results on some examples (explicit and implicit Euler schemes).


\section{Notation and preliminaries}\label{sec:prelim}
Throughout this article, $\HH$ denotes a (possibly infinite dimensional) separable Hilbert space with corresponding inner product $\langle \cdot, \cdot \rangle$ and norm $\Vert \cdot \Vert$. The set of bounded operators on $\HH$ is denoted by $\mathcal{B}(\HH)$ with the operator norm also denoted by $\|\cdot\|$, and the identity operator on $\HH$ is denoted by $\Id$. For an operator $A \in \mathcal{B}(\HH)$, its spectrum is denoted by $\sigma(A)$, its adjoint by $A^*$ and its Hermitian part by $A_H \coloneqq \frac{A + A^*}{2}$. We further denote the unit sphere in $\HH$ by $\SSS \coloneqq \{x \in \HH: \Vert x \Vert = 1\}$, the complex left half-plane by $\HHH \coloneqq \{z \in \C: \Re(z) < 0\}$, and the open unit disk by $\D \coloneqq \{z \in \C: \lvert z \rvert < 1\}$.

For the continuous-time case we recall that an operator $B\in\mathcal{B}(\HH)$ is hypocoercive if and only if $\sigma(-B)\subset\HHH$, see \cite[Th. I.3.14]{engel_nagel}, \cite[Th. 4.1]{arnold2026connectionhypocoercivityhypocontractivitycayley}, e.g.\se{, and for the discrete-time case we recall that an operator $D \in \mathcal{B}(\HH)$ is hypocontractive if and only if $\sigma(D) \subset \D$, see \cite[Proposition II. 1.3]{eisner_operator_semigroups}, e.g.}


\section{$\theta$--methods}\label{sec:theta}

For our subsequent analysis of the iteration \eqref{theta2} we first formalize the definition of $D=M_{\theta,\tau}(-B)$, where the map $M_{\theta,\tau}$ is defined as follows:
\begin{defs}
Let $0 \leq \theta \leq 1$ and $\tau > 0$. We consider the following Möbius transformations
\begin{align*}
	M_{\theta,\tau} \colon \C \setminus \left\{ \frac{1}{\theta\tau} \right\} & \longrightarrow \C, \\
	z & \longmapsto \frac{1+(1-\theta)\tau z}{1-\theta \tau z}.
\end{align*}
\end{defs}

The following lemma, needed in the subsequent analysis, follows from the proof of \cite[Theorem 4.5]{arnold2026connectionhypocoercivityhypocontractivitycayley}. We include the proof here for completeness. 

\begin{lem}\label{contractivity_rewrite}
Let $A \in \mathcal{B}(\HH)$. Then
\begin{enumerate}
\item[(i)] $A$ is semi-contractive if and only if $\Id - A^*A$ is positive semidefinite.
\item[(ii)] $A$ is contractive if and only if $\Id - A^*A$ is coercive.
\end{enumerate}
\end{lem}
\begin{proof}
For any fixed $\lambda \geq 0$ it holds:
\begin{align*}
&\Id - A^*A \geq \lambda \Id \\
\iff &\text{for all } x \in \HH: \langle (\Id - A^*A)x,x \rangle \geq \lambda \Vert x \Vert^2 \\
\iff &\text{for all } x \in \HH: \Vert x \Vert^2 - \Vert Ax \Vert^2 \geq \lambda \Vert x \Vert^2 \\
\iff &1-\lambda \geq \Vert A \Vert^2.
\end{align*}
The first claim follows from the case $\lambda=0$ in the calculation above and the second claim follows from the case $\lambda > 0$.
\end{proof}

Using the preceding lemma we can characterize (semi-)contractivity of the discrete iteration operator $M_{\theta,\tau}(A)$ in a way which will be useful for our analysis.

\begin{lem}\label{positivity_rewrite}
Let $A \in \mathcal{B}(\HH)$ \se{and $\frac{1}{\theta \tau} \notin \sigma(A)$ if $\theta > 0$}. Then
\begin{enumerate}
\item[(i)] $M_{\theta,\tau}(A)$ is semi-contractive if and only if $-2\tau A_H + \tau^2(2\theta-1)A^*A \geq 0$.
\item[(ii)] $M_{\theta,\tau}(A)$ is contractive if and only if $-2\tau A_H + \tau^2(2\theta-1)A^*A \geq \lambda \Id$ for some $\lambda > 0$.
\end{enumerate}
\end{lem}
\begin{proof}
As $\frac{1}{\theta \tau} \notin \sigma(A)$, $M_{\theta,\tau}(A)$ is well-defined as an element of $\mathcal{B}(\HH)$. Moreover, semi-contractivity (resp. contractivity) of $M_{\theta,\tau}(A)$ is equivalent to positive semidefiniteness (resp. coercivity) of $\Id - M_{\theta,\tau}(A)^*M_{\theta,\tau}(A)$ due to Lemma \ref{contractivity_rewrite}. We focus on the latter characterization in the following.
\begin{align*}
&\Id - M_{\theta,\tau}(A)^* M_{\theta,\tau}(A) \\
&= \Id - (\Id - \theta \tau A^*)^{-1} \big( \Id + (1 - \theta)\tau A^* \big)  \big( \Id + (1 - \theta)\tau A \big) (\Id - \theta \tau A)^{-1} \\
&= (\Id - \theta \tau A^*)^{-1} \big( (\Id - \theta \tau A^*)(\Id - \theta \tau A) - ( \Id + (1 - \theta)\tau A^* ) ( \Id + (1 - \theta)\tau A ) \big) (\Id - \theta \tau A)^{-1},
\end{align*}
which shows that $\Id - M_{\theta,\tau}(A)^* M_{\theta,\tau}(A)$ and $(\Id - \theta \tau A^*)(\Id - \theta \tau A) - ( \Id + (1 - \theta)\tau A^* ) ( \Id + (1 - \theta)\tau A )$ are conjugated to each other by an invertible operator. As being positive semidefinite and being coercive is invariant under such a transformation, both results follow from the equality
\begin{align*}
&(\Id - \theta \tau A^*)(\Id - \theta \tau A) - ( \Id + (1 - \theta)\tau A^* ) ( \Id + (1 - \theta)\tau A ) \\
&= (\Id - \theta \tau A^*)(\Id - \theta \tau A) - ( \Id - \theta\tau A^* + \tau A^* ) ( \Id - \theta \tau A + \tau A ) \\
&= -\tau A^* + \theta \tau^2 A^* A - \tau A + \theta \tau^2 A^* A - \tau^2 A^* A \\
&= -2\tau A_H + \tau^2(2\theta-1)A^*A.
\end{align*}
\end{proof}

We will also need the following result about hypocoercive operators.

\begin{lem}\label{b_star_b_coercive}
Let $B \in \mathcal{B}(\HH)$ be hypocoercive. Then $B^*B$ is bounded and coercive, i.e. there exist constants $\lambda \geq \mu > 0$ such that $\lambda \Id \geq B^* B \geq \mu \Id$.
\end{lem}
\begin{proof}
The only nontrivial aspect to show is coercivity of $B^*B$. Following a remark in \S\ref{sec:prelim}, it holds that $\sigma(-B) \subset \HHH$. In particular, $0 \notin \sigma(-B)=-\sigma(B)$ and so $B$ is invertible. Consequently, for any $x \in \HH$ we have
\begin{align*}
\Vert x \Vert = \Vert B^{-1} B x \Vert \leq \Vert B^{-1} \Vert \Vert Bx \Vert
\end{align*}
and hence
\begin{align*}
\langle B^*Bx,x \rangle = \langle Bx,Bx \rangle = \Vert Bx \Vert^2 \geq \Vert B^{-1} \Vert^{-2} \Vert x \Vert^2.
\end{align*}
\end{proof}

After these auxiliary results we are ready to investigate how the hypocoercivity structure of $B$ behaves under discretization using the $\theta$-methods for $\theta \neq \frac{1}{2}$. To this end, we distinguish the cases $\theta < \frac{1}{2}$ and $\theta > \frac{1}{2}$.

\begin{thm}\label{impl_results}
Let $\frac{1}{2} < \theta \leq 1$, let $B \in \mathcal{B}(\HH)$ be hypocoercive and $-B$ be semi-dissipative. Then $M_{\theta,\tau}(-B)$ is contractive for any $\tau>0$, i.e. it is hypocontractive with hypocontractivity index $0$.
\end{thm}
\begin{proof}
Lemma \ref{positivity_rewrite} yields that $M_{\theta,\tau}(-B)$ is contractive if and only if $2\tau B_H + \tau^2(2\theta-1)B^*B$ is coercive. Yet, the latter is clear because $B_H \geq 0$, $2\theta - 1 > 0$, and Lemma \ref{b_star_b_coercive} yields that $B^*B \geq \mu \Id$ for some $\mu > 0$.
\end{proof}

\begin{thm}\label{expl_results}
Let $0 \leq \theta < \frac{1}{2}$ and let $B \in \mathcal{B}(\HH)$ be hypocoercive \aaa{and $-B$ be semi-dissipative}.
\begin{enumerate}
\item[(i)] If $B$ has hypocoercivity index $0$ (i.e.\ $B$ is coercive), then the set of step sizes $\tau > 0$ such that $M_{\theta,\tau}(-B)$ is semi-contractive has the form $(0,\tau_0]$ for some $\tau_0 > 0$. Moreover, $M_{\theta,\tau}(-B)$ is even contractive, i.e. it is hypocontractive with hypocontractivity index $0$, for $0 < \tau < \tau_0$. \aaa{At the endpoint it always satisfies}\se{ $\Vert M_{\theta,\tau_0}(-B) \Vert = 1.$}
\item[(ii)] If $B$ has hypocoercivity index greater than $0$, then $M_{\theta,\tau}(-B)$ is not semi-contractive for any $\tau > 0$.
\end{enumerate}
\end{thm}
\begin{proof}
We first note that Lemma \ref{positivity_rewrite} yields that $M_{\theta,\tau}(-B)$ is semi-contractive if and only if
\begin{align*}
2\tau B_H + \tau^2(2\theta-1)B^*B \geq 0,
\end{align*}
and the latter is true if and only if 
\begin{align}\label{eq:semi-contractivity_condition}
2 B_H \geq \tau(1-2\theta)B^*B.
\end{align}
Thus, the set of all $\tau > 0$ such that $M_{\theta,\tau}(-B)$ is semi-contractive can be written equivalently as 
\begin{align*}
\mathcal{T} \coloneqq \{ \tau > 0: 2 B_H \geq \tau(1-2\theta)B^*B \}.    
\end{align*}
\begin{enumerate}
\item[(i)] 
Concerning the first part, $B$ has hypocoercivity index $0$, so $B_H$ is coercive, meaning $B_H \geq \kappa \Id$ for some $\kappa > 0$. Moreover, $B$ is hypocoercive, so Lemma \ref{b_star_b_coercive} yields $\lambda \Id \geq B^*B \geq \mu \Id$ for some $\lambda \geq \mu > 0$. Consequently, for $0 < \tau \leq \frac{2\kappa}{(1-2\theta)\lambda}$ we have
\begin{align*}
2 B_H \geq 2 \kappa \Id \geq \tau(1-2\theta)\lambda \Id \geq \tau(1-2\theta)B^*B,
\end{align*}
which shows that $\mathcal{T}$ is nonempty. Furthermore, let $\tilde{\tau} \in \mathcal{T}$ and $0 < \tau < \tilde{\tau}$. Then, using $B^*B \geq 0$, we deduce
\begin{align*}
2 B_H \geq \tilde{\tau}(1-2\theta)B^*B = \tau(1-2\theta)B^*B + (\tilde{\tau} - \tau)(1-2\theta)B^*B \geq \tau(1-2\theta)B^*B.
\end{align*}
Thus, $\tau \in \mathcal{T}$. Finally, let $\tau_n \in \mathcal{T}$ for all $n \in \N$ and $\tau_n \xrightarrow{n \to \infty} \tau$. Then for any $x \in \HH$ it holds
\begin{align*}
\langle 2B_H x, x \rangle \geq \langle \tau_n(1-2\theta)B^*B x, x \rangle.
\end{align*}
Taking the limit $n \to \infty$ we deduce
\begin{align*}
\langle 2B_H x, x \rangle \geq \langle \tau(1-2\theta)B^*B x, x \rangle ,
\end{align*}
and so $\tau \in \mathcal{T}$. Thus, the set of $\tau > 0$ such that $M_{\theta,\tau}(-B)$ is semi-contractive has the form $(0,\tau_0]$ for some $\tau_0 > 0$. Concerning the 
\aaa{contractivity statement}, Lemma \ref{positivity_rewrite}(ii) shows that
$M_{\theta,\tau}(-B)$ is contractive if and only if $2 B_H - \tau(1-2\theta)B^*B$ is coercive. The latter follows directly from semi-contractivity of $M_{\theta,\tau_0}(-B)$ (which implies that \eqref{eq:semi-contractivity_condition} holds for $\tau = \tau_0$) and the fact that $B^*B \geq \mu \Id$:
\begin{align*}
2 B_H - \tau(1-2\theta)B^*B &= 2 B_H - \tau_0(1-2\theta)B^*B + (\tau_0 - \tau)(1-2\theta)B^*B \geq (\tau_0 - \tau)(1-2\theta)B^*B \\
&\geq (\tau_0 - \tau)(1-2\theta) \mu \Id.
\end{align*}
\aaa{The last claim follows from continuity of $\Vert M_{\theta,\tau}(-B) \Vert$ w.r.t.\ $\tau$.}
\item[(ii)] Concerning the second part, we note that $B$ having hypocoercivity index greater than $0$ implies the existence of a sequence $\{x_n\} \subset \mathbb{S}$ such that $\langle B_H x_n,x_n \rangle \xrightarrow{n \to \infty} 0$. Together with the fact that $B^*B \geq \mu \Id$ for some $\mu > 0$ (which again follows from Lemma \ref{b_star_b_coercive}) we deduce
\begin{align*}
\langle (2B_H - \tau(1-2\theta)B^*B)x_n,x_n \rangle &= 2\langle B_H x_n,x_n \rangle - \tau(1-2\theta) \langle B^*B x_n, x_n \rangle \\
&\leq \langle 2B_H x_n,x_n \rangle - \tau(1-2\theta) \langle \mu x_n,x_n \rangle = \langle 2B_H x_n,x_n \rangle - \tau(1-2\theta)\mu \\
&\xrightarrow{n \to \infty} -\tau(1-2\theta)\mu < 0.
\end{align*}
Hence \eqref{eq:semi-contractivity_condition} cannot hold for any $\tau > 0$.
\end{enumerate}
\end{proof}

\se{\begin{rem}
\aaa{Theorems \ref{impl_results} and \ref{expl_results} particularly show that, for $\theta\ne\frac12$, the hypocoercivity index is never preserved unless it vanishes. This is in stark contrast to Theorem \ref{th1.5}(iii).}
\end{rem}}

\se{\begin{rem}
Theorem \ref{expl_results} discusses the effect of $\theta$-methods with $0 \leq \theta < \frac{1}{2}$ on the hypocoercivity index \aaa{and on semi-contractivity.} However, the second statement does not say anything about hypocontractivity, because a system can be hypocontractive without being semi-contractive. Indeed, following a remark in \S\ref{sec:prelim}, hypocontractivity of $M_{\theta,\tau}(-B)$ is characterized by $\sigma(M_{\theta,\tau}(-B)) \subset \D$. However, by spectral mapping, $\sigma(M_{\theta,\tau}(-B)) = M_{\theta,\tau}(\sigma(-B))$ and an elementary calculation shows that 
\begin{align*}
\lvert M_{\theta,\tau}(z) \rvert < 1 \iff 2\Re(z) \tau + (1-2\theta)\lvert z \rvert^2 \tau^2 < 0.
\end{align*}
As $\sigma(-B)$ is a bounded subset of $\HHH$, this implies immediately that for $\tau > 0$ small enough, $M_{\theta,\tau}(-B)$ will actually be hypocontractive although not being semi-contractive.
\end{rem}}

\se{\begin{rem}
Comparing Theorem \ref{impl_results} and Theorem \ref{expl_results} one can see that for $\theta > \frac{1}{2}$ \aaa{some ``coarse structure'' of the evolution} is preserved (e.g. semi-dissipative systems are mapped to semi-contractive ones), but the finer hypocoercivity index structure gets destroyed. On the other hand, for $\theta < \frac{1}{2}$ even the \aaa{mentioned ``coarse structure''} is not preserved. In view of the fact that $\frac{1}{2} \leq \theta < 1$ is exactly the range where $\theta$-methods are A-stable, this can be seen as an instance of the general principle that A-stable methods tend to be well-suited for structure-preserving discretizations. A prominent example is the preservation of maximal $L^p$-regularity as discussed in \cite{MR3582825}, \cite{MR4410757}, and \cite{MR3745012}. The hypocoercivity index is special in this regard because A-stability is not sufficient to ensure its preservation as it can be destroyed by overdamping. Thus, the midpoint method, being exactly at the borderline of A-stability, is the only $\theta$-method that preserves it.
\end{rem}}

\section{Examples}\label{sec:ex}
We first want to illustrate our results by extending \cite[Example 5.3]{arnold2026connectionhypocoercivityhypocontractivitycayley} to the implicit and explicit Euler method.

\begin{example}
We consider $\HH = \C^2$ and the matrix
\begin{align*}
B = \begin{pmatrix}
0 &\frac{1}{2} \\
-\frac{1}{2} &1
\end{pmatrix}
\end{align*}
with corresponding Hermitian part 
\begin{align*}
B_H = \begin{pmatrix}
0 &0 \\
0 &1
\end{pmatrix} \geq 0.
\end{align*}
As $B_H$ is obviously not positive definite but
\begin{align*}
B_H + B^* B_H B = 
\begin{pmatrix}
\frac{1}{4} &-\frac{1}{2}\\
-\frac{1}{2} &2
\end{pmatrix}
\end{align*}
is (because its eigenvalues are positive), B has hypocoercivity index $1$. The corresponding discrete iteration operator obtained by applying the \emph{explicit Euler method} (i.e.\ with $\theta=0$) is
\begin{align*}
D(\tau) = \begin{pmatrix}
1 &-\frac{\tau}{2} \\
\frac{\tau}{2} &1-\tau
\end{pmatrix}.
\end{align*}
It can be seen easily that $D(\tau)$ is not semi-contractive for any $\tau > 0$ because
\begin{align*}
\left\Vert 
\begin{pmatrix}
1 &-\frac{\tau}{2} \\
\frac{\tau}{2} &1-\tau
\end{pmatrix}
\begin{pmatrix}
1 \\
0
\end{pmatrix}
\right\Vert
=
\left\Vert
\begin{pmatrix}
1 \\
\frac{\tau}{2}
\end{pmatrix}
\right\Vert > 1.
\end{align*}
Hence the spectral norm of $D(\tau)$ must be strictly larger than $1$, as was predicted by Theorem \ref{expl_results} (ii). 

If we apply the \emph{implicit Euler method} (i.e.\ with $\theta=1$), we obtain
\begin{align*}
D(\tau) = 
\frac{1}{\left(1+\frac{\tau}{2}\right)^2}
\begin{pmatrix}
1+\tau &-\frac{\tau}{2} \\
\frac{\tau}{2} &1
\end{pmatrix}
\end{align*}
as the discrete iteration operator. In order to show that $D(\tau)$ is a contraction for all $\tau > 0$, we note that the spectral norm of a real valued matrix
\begin{align*}
\begin{pmatrix}
a &b \\
c &d
\end{pmatrix}
\end{align*}
is given by
\begin{align*}
\sqrt{\frac{g + \sqrt{g^2-4h}}{2}},
\end{align*}
where $g = a^2+b^2+c^2+d^2$ and $h = (ad-bc)^2$ (this follows from the fact that the spectral norm is the largest singular value). Applying this formula to $(1+\frac{\tau}{2})^2\,D(\tau)$ 
we deduce 
\begin{equation}
g 
=2\left(\left(1+\frac{\tau}{2}\right)^2 + \frac{\tau^2}{2}\right) , \qquad \label{eq:g} 
h 
=\left( 1 + \frac{\tau}{2} \right)^4 . 
\end{equation}
Consequently,
\begin{align}\label{eq:g_4h_exp}
g^2-4h 
= 4\tau^2 + 4\tau^3 + 2\tau^4.
\end{align}
Now we note that
\begin{align}\label{eq:contr_exp}
\Vert D(\tau) \Vert < 1 \iff \Vert D(\tau) \Vert^2 < 1 \iff \frac{\frac{g + \sqrt{g^2-4h}}{2}}{\left( 1+\frac{\tau}{2} \right)^4} < 1 \iff \sqrt{g^2-4h} < 2\left( 1+\frac{\tau}{2} \right)^4 - g
\end{align}
and further, using \eqref{eq:g},
\begin{align*}
2\left( 1+\frac{\tau}{2} \right)^4 - g 
= 2\tau + \frac{3\tau^2}{2} + \tau^3 + \frac{\tau^4}{8} > 0.
\end{align*}
Thus, we can square the last expression of \eqref{eq:contr_exp} again:
\begin{align*}
\sqrt{g^2-4h} < 2\left( 1+\frac{\tau}{2} \right)^4 - g \iff g^2-4h < \left( 2\tau + \frac{3\tau^2}{2} + \tau^3 + \frac{\tau^4}{8} \right)^2 ,
\end{align*}
and the latter is true because \eqref{eq:g_4h_exp} yields
\begin{align*}
g^2-4h = 4\tau^2 + 4\tau^3 + 2\tau^4 < \left( 2\tau + \frac{3\tau^2}{2} \right)^2 < \left( 2\tau + \frac{3\tau^2}{2} + \tau^3 + \frac{\tau^4}{8} \right)^2.
\end{align*}
This shows contractivity of $D(\tau)$ for any $\tau > 0$, as was predicted by Theorem \ref{impl_results}. \qed
\end{example}

As a second example we consider the case where the hypocoercivity index is $0$.

\begin{example}
We consider $\HH = \C^2$ and the matrix
\begin{align*}
B = \begin{pmatrix}
1 &\frac{1}{2} \\
-\frac{1}{2} &1
\end{pmatrix}
\end{align*}
with corresponding Hermitian part 
\begin{align*}
B_H = \begin{pmatrix}
1 &0 \\
0 &1
\end{pmatrix}.
\end{align*}
As $B_H$ is positive definite, B has hypocoercivity index $0$. The corresponding discrete iteration operator obtained by applying the \emph{explicit Euler method} is
\begin{align*}
D(\tau) = \begin{pmatrix}
1-\tau &-\frac{\tau}{2} \\
\frac{\tau}{2} &1-\tau
\end{pmatrix} .
\end{align*}
Since
\begin{align*}
D(\tau)^T D(\tau) = \left((1-\tau)^2 + \frac{\tau^2}{4}\right) \Id,
\end{align*}
we deduce that
\begin{align*}
\Vert D(\tau) \Vert = \sqrt{ (1-\tau)^2 + \frac{\tau^2}{4} }.
\end{align*}
As $f(\tau) \coloneqq (1-\tau)^2 + \frac{\tau^2}{4}$ is a parabola pointing upwards which satisfies $f(0) = 1$, $f(1) = \frac{1}{4}$, and $f \geq 0$, it is clear that the set of all $\tau > 0$ such that $f(\tau) \leq 1$ has the form $(0,\tau_0]$ for some $\tau_0 > 0$ and then the set $\{\tau>0\,\big|\,f(\tau) < 1\}$ is given by $(0,\tau_0)$. Consequently, the same holds for the set of all $\tau > 0$ where $D(\tau)$ is (semi-)contractive, confirming Theorem \ref{expl_results}. 

If we apply the \emph{implicit Euler method}, we obtain
\begin{align*}
D(\tau) = 
\begin{pmatrix}
1+\tau &\frac{\tau}{2} \\
-\frac{\tau}{2} &1+\tau
\end{pmatrix}^{-1}
= \frac{1}{(1+\tau)^2 + \frac{\tau^2}{4}}
\begin{pmatrix}
1+\tau &-\frac{\tau}{2} \\
\frac{\tau}{2} &1+\tau
\end{pmatrix}
\end{align*}
as the discrete iteration operator. Similar to before, one 
computes
\begin{align*}
\Vert D(\tau) \Vert = \frac{1}{\sqrt{ (1+\tau)^2 + \frac{\tau^2}{4}}} < 1
\end{align*}
for any $\tau > 0$, thus confirming Theorem \ref{impl_results}. \qed
\end{example}

\bibliographystyle{pamm}
\bibliography{bibliography}

\end{document}